\documentclass{article}

\usepackage[latin1]{inputenc}  
\usepackage{graphicx}          
\usepackage{psfrag}            
\usepackage{amsmath}           
\usepackage{amsthm}
\usepackage{amsfonts}
\usepackage{amssymb}
\usepackage[all]{xy}

\newcommand{\Q}{\mathbb{Q}}

\newtheorem{exm}{Example}
\newtheorem{definition}{Definition}[subsection]
\newtheorem{thm}[definition]{Theorem}

\begin{document}
\title{Dualization of projective algebraic sets by using Gröbner bases elimination techniques}
\author{{\large {\sc C{\v a}lin--{\c S}erban B{\v a}rbat}}\\
{\tt calin.barbat@web.de}}
\maketitle
\abstract{The set of common roots of a finite set $I$ (it  is an ideal) of homogeneous polynomials is known as projective algebraic set $V$. 
In this article I show how to dualize such projective algebraic sets $V$ by elimination of variables from a system of polynomials with the Gröbner bases method. A dualization algorithm is implemented in {\sc Singular} -- see \cite{GPS09}. Some examples are given. The main diagram shows the relationship between the ideal $I$, its radical $\sqrt{I}$ and their dual ideals.
}
\section{Introduction}

I read about the duality principle in the book \cite{gie} and saw some examples of dualization in the books 
\cite{gie}, \cite{kom} for quadrics and in the more recent introductory book on plane curves \cite{GF94} for 
plane curves. This last book gives some particular examples how the dualization is carried out but no general
method. Some authors mention that variables have to be eliminated from a system. For plane curves the system is 
derived nicely in \cite{bri}. For one hypersurface I recently found \cite{pw} p. 104f. But for intersections of
hypersurfaces the only example that I found was the intersection of two hypersurfaces in \cite{kom} as a general
example, without being specific. So I derived in this article the system for the intersection case.

In what follows I recommend to read \cite{GP} for the theoretical background about projective space,
homogeneous polynomial, ideal, projective variety, etc., which is not covered in the present work. For 
an introduction to Gröbner bases see \cite{cox} or \cite{fro}. I used the methods derived in this article to dualize some examples and also checked with the examples given in \cite{hr}.

\section{Motivation with plane curves}
In what follows here, we assume that the denominators do not vanish. Think of the inversion radius $r$ as having value $i=\sqrt{-1}$. (Other values are also permitted, e.g. $1$.) We consider different representations of plane curves.
\subsection{Parametric}
We consider a parametrically given plane curve $c(t)=(x(t), y(t))^t$. Then we can define the pedal curve of $c(t)$ with respect 
to the origin as
$$p(c(t))= \frac{y(t)\,x'(t) - x(t)\,y'(t)}{{x'(t)}^2 + {y'(t)}^2} \left( \begin{array}{c}-y'(t)\\x'(t)\end{array} \right)$$
(the pedal is the locus of the feet of perpendiculars from the origin to the tangents of the curve $c(t)$). We can also define 
what it means to invert $c(t)$ with respect to the circle of radius $r$ around the origin:
$$\iota(c(t))=\frac{r^2}{{x(t)}^2 + {y(t)}^2} \left( \begin{array}{c}x(t)\\y(t)\end{array} \right)$$
By composing the two maps given above we get the dual of $c(t)$ as the inverse of the pedal:
$$d(c(t))=\iota(p(c(t)))=\frac{r^2}{y(t)\,x'(t) - x(t)\,y'(t)} \left( \begin{array}{c}-y'(t)\\x'(t)\end{array} \right)$$
This is best explained by a commutative diagram:
$$
\xymatrix{ & c(t) \ar[dl]_d \ar[d]^p \\
          \iota(p(c(t))) \ar[ur] \ar[r]^\iota & p(c(t)) \ar[l] }
$$
\subsection{Complex}
Now we do the same for a curve $z(t)$ in the complex plane. The pedal is
$$p(z(t)) = \frac{\overline{z'(t)}\,z(t) - \overline{z(t)}\,z'(t)}{2\,\overline{z'(t)}} $$
The inverse is
$$\iota(z(t))=\frac{r^2}{\overline{z(t)}}$$
and the dual is (again by composition)
$$d(z(t))=\iota(p(z(t)))=\frac{2\,r^2\,z'(t)}{\overline{z(t)}\,z'(t)-\overline{z'(t)}\,z(t)}$$
A similar commutative diagram as in the parametric case holds.
\subsection{Implicit}
For implicitly given curves $f(x, y)=0$ we cannot give explicit formulas for the pedal curve but we give a method for computing it.
We need the gradient $\nabla f(x, y) = \left(\frac{\partial f}{\partial x}(x, y), \frac{\partial f}{\partial y}(x, y)\right)^t$.
Assume $p=(x,y)$ is a point of $f$ and $P=(X,Y)$ is a point of the pedal of $f$. Then, by the definition of the pedal, following
must hold:
\begin{enumerate}
	\item $p=(x,y)$ is a point of $f$: $f(x,y)=0$.
	\item $P$ lies on the tangent at $f$ through $p$: $(P-p)^t \nabla f(x, y) =0$.
	\item $P$ is orthogonal to tangent at $f$ through $p$: $P^t \left(\frac{\partial f}{\partial y}(x, y), -\frac{\partial f}{\partial x}(x, y)\right)^t =0$
\end{enumerate}
By eliminating $(x,y)$ from these three equations we get an equation in $(X,Y)$ which is the pedal curve. For convenience we substitute $(X,Y)\mapsto (x,y)$. In what follows, we will see how the elimination can be done with Gröbner bases.

The inverse of $f(x, y)=0$ is $f\left(\frac{r^2\,x}{x^2 + y^2},\frac{r^2\,y}{x^2 + y^2}\right)=0$.
The dual of $f$ is the composition of inversion and pedal as constructed above.

\section{Theory}
\subsection{Case of one homogeneous polynomial}

First we consider the following projective algebraic set
$$
V(p) = \{{\bf x} \in \mathbb{K}^{n+1}\mid p({\bf x})=0\}
$$
with $\mathbb{K}$ an algebraic closed field, ${\bf x}=(x_0, x_1, \ldots, x_n)$ a point of $\mathbb{K}^{n+1}$
and $p$ a homogeneous polynomial from $\mathbb{K}[x_0, x_1, \ldots, x_n]$. $V(p)$ consists of all roots ${\bf x}=(x_0, x_1, \ldots, x_n)$ of $p$ and is a hypersurface in the projective space ${\mathbb P}^n$. 

Let ${\bf u} = (u_0, u_1, \ldots, u_n)$ be a normal vector to $V(p)$ in a regular point ${\bf x} \in V(p)$.
On the other side we know that the gradient 
$$\nabla p({\bf x}) = \left(\frac{\partial p}{\partial x_0}({\bf x}), \frac{\partial p}{\partial x_1}({\bf x}), \ldots, \frac{\partial p}{\partial x_n}({\bf x})\right)^t$$
is normal to $V(p)$ in $\bf x$. Therefore ${\bf u}$ and $\nabla p({\bf x})$ are linearly dependent. This can be written as ${\bf u}=\lambda\nabla p({\bf x})$ with a factor $\lambda$. We can form the following system
\begin{align}
\left\{\begin{aligned}
\begin{split}
p({\bf x}) &= 0 \\
{\bf u} - \lambda \nabla p({\bf x}) &= {\bf 0} \\
\end{split}
\end{aligned}\right.\label{ds1}
\end{align}
\begin{definition}
We define the set $V^*(p)=\{{\bf u} \in \mathbb{K}^{n+1}\mid p({\bf x}) = 0, \, {\bf u} - \lambda \nabla p({\bf x}) = 0 \}$
of partial solutions to the system (\ref{ds1}) to be the dual of $V(p)$.
\end{definition}
Note that we are not interested in a complete solution of (\ref{ds1}), but only in the partial solutions which I call here the $\bf u$-part of the solution. The $\bf u$-part of the solution of this system is the result of applying the Gauß map to $p({\bf x}) = 0$, where the Gauß map (see \cite{s}, p. 103) is given -- in Chow coordinates (see \cite{l}) -- by
$$
\gamma : {\bf x} \mapsto {\bf u} = \lambda \nabla p({\bf x})
$$
Now we want to construct an equivalent system to (\ref{ds1}) but simpler in structure, describing the same dual algebraic set $V^*(p)$.
\begin{thm}
There exists a system $B$ of polynomials in $\mathbb{K}[u_0, u_1, \ldots, u_n]$ with the same solution set $V^*(p)$ as the system (\ref{ds1}).
\end{thm}
\begin{proof}
We start with the system (\ref{ds1}) viewed as system of polynomials $$q_j \in \mathbb{K}[x_0, x_1, \ldots, x_n, \lambda, u_0, u_1, \ldots, u_n]$$
and eliminate the first $n+2$ variables. The system is:
\begin{align*}
\left\{\begin{aligned}
\begin{split}
q_1(x_0, x_1, \ldots, x_n, \lambda, u_0, u_1, \ldots, u_n) &= p(x_0, \ldots, x_n) = 0 \\
q_2(x_0, x_1, \ldots, x_n, \lambda, u_0, u_1, \ldots, u_n) &= u_0 - {\lambda \frac{\partial p}{\partial x_0}(x_0, \ldots, x_n)} = 0 \\
&\vdots \\
q_{i+2}(x_0, x_1, \ldots, x_n, \lambda, u_0, u_1, \ldots, u_n) &= u_i - {\lambda \frac{\partial p}{\partial x_i}(x_0, \ldots, x_n)} = 0 \\
&\vdots \\
q_{n+2}(x_0, x_1, \ldots, x_n, \lambda, u_0, u_1, \ldots, u_n) &= u_n - {\lambda \frac{\partial p}{\partial x_n}(x_0, \ldots, x_n)} = 0 \\
\end{split}
\end{aligned}\right. 
\end{align*}
Let $G$ be a Gröbner basis for the ideal $S = (q_1, \ldots, q_j, \ldots, q_{n+3})$ with respect to an elimination ordering, where e.g. 
${\bf x} > {\bf \lambda} > {\bf u}$. By the Elimination Theorem of \cite{cox} this basis $G$ eliminates ${\bf x}$ and ${\bf \lambda}$
and $B=G \cap \mathbb{K}[u_0, u_1, \ldots, u_n]$ is a Gröbner basis of the elimination ideal $E=S \cap \mathbb{K}[u_0, u_1, \ldots, u_n]$.
We get $E=(B)$.
By construction, this basis $B$ is a system of polynomials in $\mathbb{K}[u_0, u_1, \ldots, u_n]$ having the same partial solution set $V^*(p)$ as the 
original system (\ref{ds1}).
\end{proof}
Because of this property we define $E=(B)$ to be the dual ideal of the initial ideal $(p)$. 
With the canonical isomorphism $u_i \stackrel{\cong}{\mapsto} x_i$ we can map the ideal $E \in \mathbb{K}[{\bf u}]$ to an ideal 
$D \in \mathbb{K}[{\bf x}]$ which leads to the following commutative diagram:
$$\xymatrix{
          (p) \in \mathbb{K}[{\bf x}] \ar[r]^{\gamma} \ar[d] & E \in \mathbb{K}[{\bf u}] \ar[dl]^{\cong} \\ 
           D  \in \mathbb{K}[{\bf x}] 
}$$
\begin{exm} 
As an example we dualize the quadric $(n=3)$:
$$
-b_{0}x_{0}^{2}+2b_{1}x_{0}x_{1}-\frac{b_{1}-1}{a_{1}}x_{1}^{2}+\frac{1}{a_{2}}x_{2}^{2}+\frac{1}{a_{3}}x_{3}^{2} = 0
$$
The system (\ref{ds1}) is here:
$$\left\{
\begin{aligned}
-b_{0}x_{0}^{2}+2b_{1}x_{0}x_{1}-\frac{b_{1}-1}{a_{1}}x_{1}^{2}+\frac{1}{a_{2}}x_{2}^{2}+\frac{1}{a_{3}}x_{3}^{2} &=0 \\
2b_{0}x_{0}\lambda_{1}-2b_{1}x_{1}\lambda_{1}+u_{0} &=0 \\
-2b_{1}x_{0}\lambda_{1}+\frac{2b_{1}-2}{a_{1}}x_{1}\lambda_{1}+u_{1} &=0 \\
-\frac{2}{a_{2}}x_{2}\lambda_{1}+u_{2} &=0 \\
-\frac{2}{a_{3}}x_{3}\lambda_{1}+u_{3} &=0 
\end{aligned}
\right.$$
By eliminating $x_i$ and $\lambda$ from these equations we get one equation in $u_i$, which is the dual quadric
$$
\begin{array}{c}
(b_{1}-1)u_{0}^{2}+2a_{1}b_{1}u_{0}u_{1}+a_{1}b_{0}u_{1}^{2}+(a_{1}a_{2}b_{1}^{2}-a_{2}b_{0}b_{1}+a_{2}b_{0})u_{2}^{2}+(a_{1}a_{3}b_{1}^{2}-a_{3}b_{0}b_{1}+a_{3}b_{0})u_{3}^{2} = 0
\end{array}
$$
\end{exm}
\begin{exm} 
As another example we dualize a quadric over $\Q[w,x,y,z]$:
$$\left(
\begin{array}{c}
w^{2}-x^{2}
\end{array}
\right)$$
and get
$$\left(
\begin{array}{c}
z, \\
y, \\
w^{2}-x^{2}
\end{array}
\right)$$
How about dualizing this result? This simple example shows that we must be able to dualize ideals given by more than one polynomial. This is done in the next section.
\end{exm}
\subsection{Case of a system of homogeneous polynomials}
The same method can be extended and applied to finite sets of polynomials which represent geometrically their intersection as hypersurfaces.
We consider the projective algebraic set
$$
V := V(I) = V(p_1, p_2, \ldots, p_m) = \{{\bf x} \in \mathbb{K}^{n+1}\mid p_1({\bf x})=\ldots=p_m({\bf x})=0\}
$$
with $\mathbb{K}$ an algebraically closed field, ${\bf x}=(x_0, x_1, \ldots, x_n)$ a point of $\mathbb{K}^{n+1}$
and $I$ the ideal generated by $m \geq 1$ homogeneous polynomials $p_1, p_2, \ldots, p_m$ from 
$\mathbb{K}[{\bf x}]$.

The dual projective algebraic set $V^*$ of $V$ is also a projective algebraic set which is the zero set of an ideal $E$ generated
by polynomials from $\mathbb{K}[u_0, u_1, \ldots, u_n]$. We get this elimination ideal $E$ by eliminating (using a suitable Gröbner 
basis) the $m+n+1$ variables $x_i$ and $\lambda_j$ from the $m+n+1$ equations:
\begin{align}
\left\{\begin{aligned}
\begin{split}
p_1(x_0, \ldots, x_n) &= 0 \\
&\vdots \\
p_m(x_0, \ldots, x_n) &= 0 \\
u_0 - \sum_{j=1}^{m}{\lambda_j \frac{\partial p_j}{\partial x_0}(x_0, \ldots, x_n)} &= 0 \\
&\vdots \\
u_i - \sum_{j=1}^{m}{\lambda_j \frac{\partial p_j}{\partial x_i}(x_0, \ldots, x_n)} &= 0 \\
&\vdots \\
u_n - \sum_{j=1}^{m}{\lambda_j \frac{\partial p_j}{\partial x_n}(x_0, \ldots, x_n)} &= 0 \\
\end{split}
\end{aligned}\right. \label{ds}
\end{align}
The $x_i$ are point and the $u_i$ hyperplane coordinates but their roles can be interchanged.

With the Jacobi matrix of ${\bf p}({\bf x})=(p_1, p_2, \ldots, p_m)^t({\bf x})$:
$$
J =  \begin{pmatrix}
\frac{\partial p_1}{\partial x_0} & \frac{\partial p_1}{\partial x_1} & \ldots & \frac{\partial p_1}{\partial x_n} \\
\vdots & \vdots & \ddots & \vdots & \\
\frac{\partial p_m}{\partial x_0} & \frac{\partial p_m}{\partial x_1} & \ldots & \frac{\partial p_m}{\partial x_n} \end{pmatrix}.
$$
and setting ${\bf \lambda}=(\lambda_1, \lambda_2, \ldots, \lambda_m)^t$
we can write the system (\ref{ds}) vectorially as
\begin{align}
\left\{\begin{aligned}
{\bf p}({\bf x})&={\bf 0} \\
{\bf u} - \sum_{j=1}^{m}{\lambda_j \nabla p_j({\bf x})} &= {\bf u}-J^t({\bf x}) {\bf \lambda} = {\bf 0} \\
\end{aligned}\right.
\end{align}
For $m=1$ this is the same as system (\ref{ds1}).
We redefine the Gauß map:
$$
\gamma : {\bf x} \mapsto {\bf u} = J^t({\bf x}) {\bf \lambda}
$$
\begin{exm} 
We dualize now the following ideal
$$\left(
\begin{array}{c}
x_{2}^{2}-x_{3}^{2}, \\
x_{0}-x_{2}
\end{array}
\right)$$
The system (\ref{ds}) is here (written as ideal)
$$\left(
\begin{array}{c}
x_{2}^{2}-x_{3}^{2}, \\
x_{0}-x_{2}, \\
-\lambda_{2}+u_{0}, \\
u_{1}, \\
-2x_{2}\lambda_{1}+\lambda_{2}+u_{2}, \\
2x_{3}\lambda_{1}+u_{3}
\end{array}
\right)$$
A sorted Gröbner basis with the elimination property is
$$\left(
\begin{array}{c}
u_{1}, \\
u_{0}^{2}+2u_{0}u_{2}+u_{2}^{2}-u_{3}^{2}, \\
\lambda_{2}-u_{0}, \\
2x_{3}\lambda_{1}+u_{3}, \\
x_{2}u_{3}+x_{3}u_{0}+x_{3}u_{2}, \\
x_{2}u_{0}+x_{2}u_{2}+x_{3}u_{3}, \\
2x_{2}\lambda_{1}-u_{0}-u_{2}, \\
x_{2}^{2}-x_{3}^{2}, \\
x_{0}-x_{2}
\end{array}
\right)$$
and we see that the first two elements generate the elimination ideal $E$ for this example
$$\left(
\begin{array}{c}
u_{1}, \\
u_{0}^{2}+2u_{0}u_{2}+u_{2}^{2}-u_{3}^{2}
\end{array}
\right)$$
\end{exm}
\subsection{Main diagram}
We have, by denoting with $D(\sqrt{I})$ the dual of the radical ideal of $I$:
$$\xymatrix{
  \sqrt{I} \ar[d]_{\gamma} & & I \ar[d]_<<<<{\gamma}\ar[ll]_{\supseteq} & \\
  D(\sqrt{I}) \ar[d]\ar[r]^{\subseteq} & \sqrt{D(\sqrt{I})} \ar[d]\ar@/^1.5pc/[rr]^>>>>{\subseteq}|\hole & D(I) \ar[d]\ar[r]^{\subseteq} & \sqrt{D(I)} \ar[d]\\
  V(D(\sqrt{I})) & V(\sqrt{D(\sqrt{I})}) \ar[l]_{\supseteq} & V(D(I)) & V(\sqrt{D(I)}) \ar[l]_{\supseteq} \ar@/^1.5pc/[ll]^>>>>{\supseteq}
}$$
(The bent arrow in the middle of the diagram needs a proof.)
As an example explaining this diagram consider:
$$I=\left(
\begin{array}{c}
z^{2}, \\
x+y-z
\end{array}
\right)$$
Then we have
$$D(\sqrt{I})=\sqrt{D(\sqrt{I})}=\left(
\begin{array}{c}
x-y
\end{array}
\right)$$
$$D(I)=\left(
\begin{array}{c}
x-y, \\
y^{2}+2yz+z^{2}
\end{array}
\right)$$
$$\sqrt{D(I)}=\left(
\begin{array}{c}
y+z, \\
x-y
\end{array}
\right)$$
By adding one polynom we get another example:
$$I=\left(
\begin{array}{c}
z^{2}, \\
x+y-z, \\
x^{3}+y^{3}
\end{array}
\right)$$
$$D(\sqrt{I})=\sqrt{D(\sqrt{I})}=\left(
\begin{array}{c}
x-y
\end{array}
\right)$$
$$D(I)=\left(
\begin{array}{c}
xy-y^{2}+xz-yz, \\
x^{2}-2xy+y^{2}
\end{array}\right)=\left(
\begin{array}{c}
(x-y)(y+z), \\
(x-y)^{2}
\end{array}
\right)$$
$$\sqrt{D(I)}=\left(
\begin{array}{c}
x-y
\end{array}
\right)$$

\section{Examples}
\subsection{Steiner's Roman surface}
I wrote a {\sc Singular} procedure {\tt dual} using the fast elimination provided by {\sc Singular} with the combination of {\tt hilb} and {\tt eliminate} for computing the dual ideal $D$ of a given ideal $I$. The procedure {\tt dual} takes as argument a finitely generated homogeneous ideal $I=\left( p_1, \ldots, p_m \right)$, where all generating polynomials $p_k$ are homogeneous and elements of the {\tt basering}.

A first example {\sc Singular} input file for testing the functionality is given in the appendix.
This file dualizes Steiner's Roman surface (see \cite{hr} for other nice dualization examples):
\begin{exm} 
\begin{enumerate}
	\item First we set the {\tt basering} to $\Q[x_{0},x_{1},x_{2},x_{3}]$.
	\item As an example we want to dualize the ideal 
$$I=\left(
\begin{array}{c}
x_{1}^{2}x_{2}^{2}-x_{0}x_{1}x_{2}x_{3}+x_{1}^{2}x_{3}^{2}+x_{2}^{2}x_{3}^{2}
\end{array}
\right)$$	
  \item We call procedure {\tt dual} on this ideal and the first thing which it does is to adjoin auxiliary variables to our ring. The new ring is: $$\Q[x_{0},x_{1},x_{2},x_{3},\lambda_{1},u_{0},u_{1},u_{2},u_{3}]$$
  \item The next step is constructing the following ideal generated by the system (\ref{ds}) of equations, which in this case is
$$S=\left(
\begin{array}{c}
x_{1}^{2}x_{2}^{2}-x_{0}x_{1}x_{2}x_{3}+x_{1}^{2}x_{3}^{2}+x_{2}^{2}x_{3}^{2}, \\
x_{1}x_{2}x_{3}\lambda_{1}+u_{0}, \\
-2x_{1}x_{2}^{2}\lambda_{1}+x_{0}x_{2}x_{3}\lambda_{1}-2x_{1}x_{3}^{2}\lambda_{1}+u_{1}, \\
-2x_{1}^{2}x_{2}\lambda_{1}+x_{0}x_{1}x_{3}\lambda_{1}-2x_{2}x_{3}^{2}\lambda_{1}+u_{2}, \\
x_{0}x_{1}x_{2}\lambda_{1}-2x_{1}^{2}x_{3}\lambda_{1}-2x_{2}^{2}x_{3}\lambda_{1}+u_{3}
\end{array}
\right)$$
  \item The Gr{\"o}bner basis for $S$ w.r.t. the elimination order is (we only show some elements at the beginning and the end here because it is bigger than this page)
$$G=
\left(
\begin{array}{c}
4u_{0}^{3}-u_{0}u_{1}^{2}-u_{0}u_{2}^{2}+u_{1}u_{2}u_{3}-u_{0}u_{3}^{2}, \\
x_{3}^{3}\lambda_{1}u_{0}u_{1}^{2}+x_{3}^{3}\lambda_{1}u_{0}u_{2}^{2}-x_{3}^{3}\lambda_{1}u_{1}u_{2}u_{3}+u_{0}^{2}u_{1}u_{2}-2u_{0}^{3}u_{3}, \\
4x_{3}^{3}\lambda_{1}u_{0}^{2}-x_{3}^{3}\lambda_{1}u_{3}^{2}+u_{0}u_{1}u_{2}-2u_{0}^{2}u_{3}, \\
\cdots, \\
x_{0}^{2}x_{2}x_{3}\lambda_{1}-4x_{1}^{2}x_{2}x_{3}\lambda_{1}-4x_{2}^{3}x_{3}\lambda_{1}-2x_{0}x_{1}x_{3}^{2}\lambda_{1}+x_{0}u_{1}+2x_{2}u_{3}, \\
x_{0}^{2}x_{1}x_{3}\lambda_{1}-4x_{1}^{3}x_{3}\lambda_{1}-4x_{1}x_{2}^{2}x_{3}\lambda_{1}-2x_{0}x_{2}x_{3}^{2}\lambda_{1}+x_{0}u_{2}+2x_{1}u_{3}
\end{array}
\right)$$
  \item The elimination ideal $E$ consists here only of the first element of $G$
$$E=\left(
\begin{array}{c}
4u_{0}^{3}-u_{0}u_{1}^{2}-u_{0}u_{2}^{2}+u_{1}u_{2}u_{3}-u_{0}u_{3}^{2}
\end{array}
\right)$$
   This is the dual ideal of $I$. We can interpret the $x$ as point coordinates and the $u$ as hyperplane coordinates. But we want to be able to pass this ideal $E$ to {\tt dual} and dualize it too! For this to work, we have to make one final step in {\tt dual} and this is to map the $x$s to the $u$s and viceversa, giving as result the dual ideal of $I$ in point coordinates, which is
$$D=\left(
\begin{array}{c}
4x_{0}^{3}-x_{0}x_{1}^{2}-x_{0}x_{2}^{2}+x_{1}x_{2}x_{3}-x_{0}x_{3}^{2}
\end{array}
\right)$$     
  \item When calling {\tt dual} on $D$ we get $I$ as dual ideal of $D$. We won't show here the steps, but the reader can generate them using the example input file from the appendix.
\end{enumerate}
The input file to {\sc Singular} for this example generates the figure \ref{fig:steiner} by using {\tt surfex.lib} which in turn uses the program {\tt surf}.
\begin{figure}[htbp]
	\centering
		\includegraphics[width=8cm]{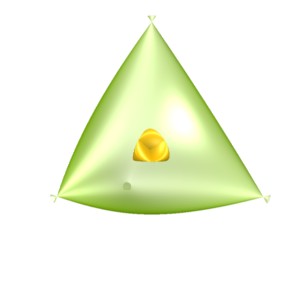}
	\caption{Steiner surface and its dual surface}
	\label{fig:steiner}
\end{figure}
\end{exm}

\subsection{Parametrized quadric}
We present now another example, the forth and back dualization of a somewhat special quadric.
\begin{exm} 
\begin{enumerate}
	\item Our ring is at the beginning $\Q(a_{1},a_{2},a_{3},b_{0},b_{1})[x_{0},x_{1},x_{2},x_{3}]$.
  \item The quadric is given by the zero set of the following ideal (this zero set is a projective algebraic set)
$$Q=\left(
\begin{array}{c}
-b_{0}x_{0}^{2}+2b_{1}x_{0}x_{1}-\frac{b_{1}-1}{a_{1}}x_{1}^{2}+\frac{1}{a_{2}}x_{2}^{2}
\end{array}
\right)$$
  \item We adjoin auxiliary variables and the new ring is $$\Q(a_{1},a_{2},a_{3},b_{0},b_{1})[x_{0},x_{1},x_{2},x_{3},\lambda_{1},u_{0},u_{1},u_{2},u_{3}]$$
  \item Here we construct the system (\ref{ds}) of equations
$$S=\left(
\begin{array}{c}
-b_{0}x_{0}^{2}+2b_{1}x_{0}x_{1}-\frac{b_{1}-1}{a_{1}}x_{1}^{2}+\frac{1}{a_{2}}x_{2}^{2}, \\
2b_{0}x_{0}\lambda_{1}-2b_{1}x_{1}\lambda_{1}+u_{0}, \\
-2b_{1}x_{0}\lambda_{1}+\frac{2b_{1}-2}{a_{1}}x_{1}\lambda_{1}+u_{1}, \\
-\frac{2}{a_{2}}x_{2}\lambda_{1}+u_{2}, \\
u_{3}
\end{array}
\right)$$
  \item The Gr{\"o}bner basis of $S$ is
$$G=\left(
\begin{array}{c}
u_{3}, \\
(b_{1}-1)u_{0}^{2}+2a_{1}b_{1}u_{0}u_{1}+a_{1}b_{0}u_{1}^{2}+(a_{1}a_{2}b_{1}^{2}-a_{2}b_{0}b_{1}+a_{2}b_{0})u_{2}^{2}, \\
2x_{2}\lambda_{1}-a_{2}u_{2}, \\
(a_{1}a_{2}b_{1}^{2}-a_{2}b_{0}b_{1}+a_{2}b_{0})x_{1}u_{2}-a_{1}b_{1}x_{2}u_{0}-a_{1}b_{0}x_{2}u_{1}, \\
(2a_{1}b_{1}^{2}-2b_{0}b_{1}+2b_{0})x_{1}\lambda_{1}-a_{1}b_{1}u_{0}-a_{1}b_{0}u_{1}, \\
a_{2}b_{0}x_{0}u_{2}-a_{2}b_{1}x_{1}u_{2}+x_{2}u_{0}, \\
a_{1}b_{0}x_{0}u_{1}-(b_{1}-1)x_{1}u_{0}-2a_{1}b_{1}x_{1}u_{1}-a_{1}b_{1}x_{2}u_{2}, \\
x_{0}u_{0}+x_{1}u_{1}+x_{2}u_{2}, \\
2b_{0}x_{0}\lambda_{1}-2b_{1}x_{1}\lambda_{1}+u_{0}, \\
a_{1}a_{2}b_{0}x_{0}^{2}-2a_{1}a_{2}b_{1}x_{0}x_{1}+(a_{2}b_{1}-a_{2})x_{1}^{2}-a_{1}x_{2}^{2}
\end{array}
\right)$$
  \item The elimination ideal (the dual of $Q$) is
$$E=\left(
\begin{array}{c}
(b_{1}-1)u_{0}^{2}+2a_{1}b_{1}u_{0}u_{1}+a_{1}b_{0}u_{1}^{2}+(a_{1}a_{2}b_{1}^{2}-a_{2}b_{0}b_{1}+a_{2}b_{0})u_{2}^{2}, \\
u_{3}
\end{array}
\right)$$
  \item We substitute $u \mapsto x$ in $E$ and get 
$$D=\left(
\begin{array}{c}
(b_{1}-1)x_{0}^{2}+2a_{1}b_{1}x_{0}x_{1}+a_{1}b_{0}x_{1}^{2}+(a_{1}a_{2}b_{1}^{2}-a_{2}b_{0}b_{1}+a_{2}b_{0})x_{2}^{2}, \\
x_{3}
\end{array}
\right)$$
\end{enumerate}
\end{exm}
This ideal $D$ can be now again dualized which constitutes our next example.
\begin{exm}
\begin{enumerate}
	\item After adjoining the auxiliary variables, our ring is $$\Q(a_{1},a_{2},a_{3},b_{0},b_{1})[x_{0},x_{1},x_{2},x_{3},\lambda_{1},\lambda_{2},u_{0},u_{1},u_{2},u_{3}]$$
	Notice here that -- since our ideal $D$ has two generators -- we have to adjoin two $\lambda$s (the procedures do this automatically).
	\item The system (\ref{ds}) -- in fact an ideal too -- is here
$$S_2=\left(
\begin{array}{c}
(b_{1}-1)x_{0}^{2}+2a_{1}b_{1}x_{0}x_{1}+a_{1}b_{0}x_{1}^{2}+(a_{1}a_{2}b_{1}^{2}-a_{2}b_{0}b_{1}+a_{2}b_{0})x_{2}^{2}, \\
x_{3}, \\
-(2b_{1}-2)x_{0}\lambda_{1}-2a_{1}b_{1}x_{1}\lambda_{1}+u_{0}, \\
-2a_{1}b_{1}x_{0}\lambda_{1}-2a_{1}b_{0}x_{1}\lambda_{1}+u_{1}, \\
-(2a_{1}a_{2}b_{1}^{2}-2a_{2}b_{0}b_{1}+2a_{2}b_{0})x_{2}\lambda_{1}+u_{2}, \\
-\lambda_{2}+u_{3}
\end{array}
\right)$$
  \item The corresponding Gr{\"o}bner basis is 
$$G_2=\left(
\begin{array}{c}
a_{1}a_{2}b_{0}u_{0}^{2}-2a_{1}a_{2}b_{1}u_{0}u_{1}+(a_{2}b_{1}-a_{2})u_{1}^{2}-a_{1}u_{2}^{2}, \\
\lambda_{2}-u_{3}, \\
x_{3}, \\
(2a_{1}a_{2}b_{1}^{2}-2a_{2}b_{0}b_{1}+2a_{2}b_{0})x_{2}\lambda_{1}-u_{2}, \\
a_{1}x_{1}u_{2}-a_{1}a_{2}b_{1}x_{2}u_{0}+(a_{2}b_{1}-a_{2})x_{2}u_{1}, \\
(2a_{1}^{2}b_{1}^{2}-2a_{1}b_{0}b_{1}+2a_{1}b_{0})x_{1}\lambda_{1}-a_{1}b_{1}u_{0}+(b_{1}-1)u_{1}, \\
x_{0}u_{2}+a_{2}b_{0}x_{2}u_{0}-a_{2}b_{1}x_{2}u_{1}, \\
(b_{1}-1)x_{0}u_{1}-a_{1}b_{0}x_{1}u_{0}+2a_{1}b_{1}x_{1}u_{1}+a_{1}b_{1}x_{2}u_{2}, \\
x_{0}u_{0}+x_{1}u_{1}+x_{2}u_{2}, \\
(2b_{1}-2)x_{0}\lambda_{1}+2a_{1}b_{1}x_{1}\lambda_{1}-u_{0}, \\
(b_{1}-1)x_{0}^{2}+2a_{1}b_{1}x_{0}x_{1}+a_{1}b_{0}x_{1}^{2}+(a_{1}a_{2}b_{1}^{2}-a_{2}b_{0}b_{1}+a_{2}b_{0})x_{2}^{2}
\end{array}
\right)$$
  \item The new elimination ideal is 
$$\left(
\begin{array}{c}
a_{1}a_{2}b_{0}u_{0}^{2}-2a_{1}a_{2}b_{1}u_{0}u_{1}+(a_{2}b_{1}-a_{2})u_{1}^{2}-a_{1}u_{2}^{2}
\end{array}
\right)$$
  \item After mapping the $x$s and $u$s we finally get $Q$ again
$$\left(
\begin{array}{c}
a_{1}a_{2}b_{0}x_{0}^{2}-2a_{1}a_{2}b_{1}x_{0}x_{1}+(a_{2}b_{1}-a_{2})x_{1}^{2}-a_{1}x_{2}^{2}
\end{array}
\right)$$
One might wonder why this looks different than the initial equation for $Q$, but if you multiply the original equation with the non-zero factor $- a_{1}a_{2} \neq 0$ -- an operation which does not change the zero set, you see that you get the same quadric.
\end{enumerate}
\end{exm}
\subsection{8-shaped space curve}
\begin{exm}
In this example we intersect a sphere and a cylinder to get an 8-shaped space curve, which is given by the ideal:
$$I = \left(
\begin{array}{c}
x^{2}+y^{2}-1, \\
x^{2}+y^{2}+z^{2}-2x-3
\end{array}
\right)$$
You can see these surfaces and their intersection in figure \ref{fig:8begin}.

Now we dualize the ideal $I$ and get the following ideal (after dehomogenizing) with one polynomial as generator:
$$D = \left(
\begin{array}{c}
4x^{6}+12x^{4}y^{2}+12x^{2}y^{4}+4y^{6}-12x^{4}z^{2}-24x^{2}y^{2}z^{2}-12y^{4}z^{2}-15x^{2}z^{4}+\\
12y^{2}z^{4}-4z^{6}+36x^{3}z^{2}+36xy^{2}z^{2}+18xz^{4}-8x^{4}-16x^{2}y^{2}-8y^{4}-20x^{2}z^{2}-\\
20y^{2}z^{2}+z^{4}-4xz^{2}+4x^{2}+4y^{2}
\end{array}
\right)$$
$D$ is the dual of the 8-shaped space curve. They are depicted in figure \ref{fig:8dual}.
\begin{figure}[htbp]
	\centering
		\includegraphics[width=8cm]{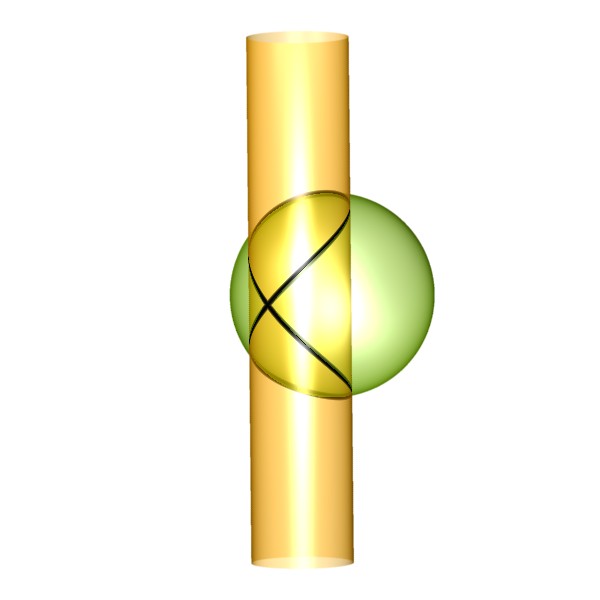}
	\caption{8-shaped curve as intersection of cylinder and sphere}
	\label{fig:8begin}
\end{figure}
\begin{figure}[htbp]
	\centering
		\includegraphics[width=8cm]{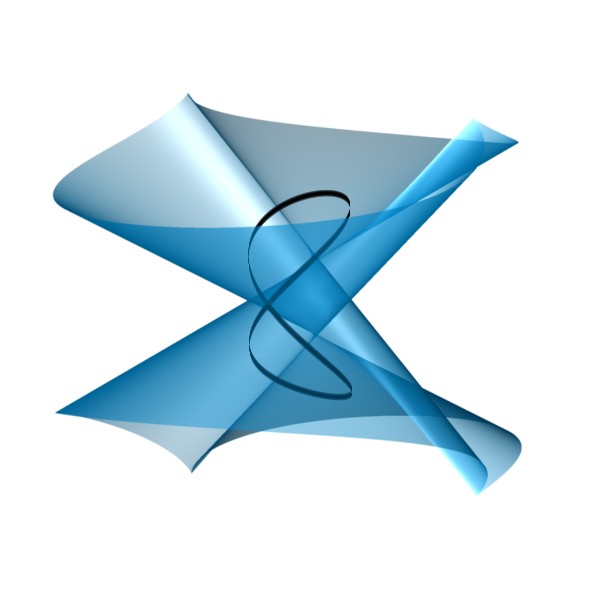}
	\caption{8-shaped curve and its dual surface}
	\label{fig:8dual}
\end{figure}
\end{exm}
The {\sc Singular} code for this example is:
\begin{verbatim}
// load libs
LIB "duality.lib";
LIB "surfex.lib";

ring r=0,(t,x,y,z),dp; // ring over Q
short = 0; // print polynomials with ^

// two polynomials 
poly cylinder = x^2+y^2-1;
poly sphere = (x-1)^2+y^2+z^2-2^2;

// intersection ideal (a space curved shaped like an 8)
ideal i1 = cylinder, sphere;   // inhomogeneous for plotting
ideal i2 = homog(i1, t);       // homogeneous for dualising
ideal d1 = dual(i2);           // dual of intersection ideal (a surface)
poly  d2 = subst(d1[1], t, 1); // dehomogenize dual for plotting
d2;                            // show result 

// plot everything
// (in surfex you may want to set transparency options for some surfaces)
plotRotatedList(list(cylinder, sphere, i1, d2), list(x,y,z)); 
\end{verbatim}
\subsection{Examples from the introductory book \cite{GF94}}
As further examples that the procedure {\tt dual} works, I dualize some planar algebraic curves from the book \cite{GF94} without including the intermediary output from {\sc Singular}.
\begin{exm} 
The Neil parabola is given by
$$\left(
\begin{array}{c}
x_{1}^{3}-x_{0}x_{2}^{2}
\end{array}
\right)$$
and its dual is
$$\left(
\begin{array}{c}
4x_{1}^{3}+27x_{0}x_{2}^{2}
\end{array}
\right)$$
You can see an illustration in figure \ref{fig:ex05_neil_parabola}
\begin{figure}[htbp]
	\centering
		\includegraphics[width=4cm]{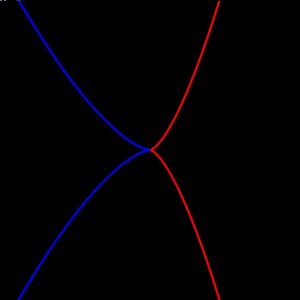}
	\caption{Neil parabola (red) and its dual}
	\label{fig:ex05_neil_parabola}
\end{figure}
\end{exm}

\begin{exm} 
The Newton knot is given by
$$\left(
\begin{array}{c}
x_{0}x_{1}^{2}+x_{1}^{3}-x_{0}x_{2}^{2}
\end{array}
\right)$$
and its dual cardioid is 
$$\left(
\begin{array}{c}
4x_{0}x_{1}^{3}-4x_{1}^{4}+27x_{0}^{2}x_{2}^{2}-36x_{0}x_{1}x_{2}^{2}+8x_{1}^{2}x_{2}^{2}-4x_{2}^{4}
\end{array}
\right)$$
You can see an illustration in figure \ref{fig:ex06_newton_knot}
\begin{figure}[htbp]
	\centering
		\includegraphics[width=4cm]{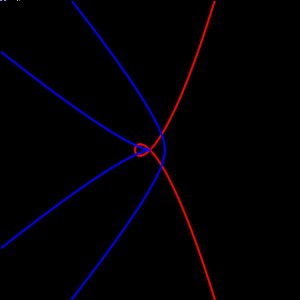}
	\caption{Newton knot (red) and its dual cardioid}
	\label{fig:ex06_newton_knot}
\end{figure}
\end{exm}

\begin{exm} 
The hypocycloid is given by
$$\left(
\begin{array}{c}
x_{0}^{2}x_{1}^{2}-2x_{0}^{2}x_{1}x_{2}-2x_{0}x_{1}^{2}x_{2}+x_{0}^{2}x_{2}^{2}-2x_{0}x_{1}x_{2}^{2}+x_{1}^{2}x_{2}^{2}
\end{array}
\right)$$
and its dual is calculated as
$$\left(
\begin{array}{c}
x_{0}^{3}+3x_{0}^{2}x_{1}+3x_{0}x_{1}^{2}+x_{1}^{3}+3x_{0}^{2}x_{2}-21x_{0}x_{1}x_{2}+3x_{1}^{2}x_{2}+3x_{0}x_{2}^{2}+3x_{1}x_{2}^{2}+x_{2}^{3}
\end{array}
\right)$$
You can see an illustration in figure \ref{fig:ex07_hypocycloid}
\begin{figure}[htbp]
	\centering
		\includegraphics[width=4cm]{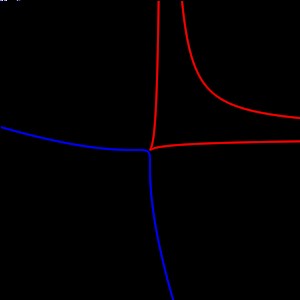}
	\caption{Hypocycloid (red) and its dual}
	\label{fig:ex07_hypocycloid}
\end{figure}
\end{exm}

\begin{exm} 
The next example was inspired by the Klein quartic. I changed one of the ellipses to a hyperbola and
the generated ideal is 
$$\left(
\begin{array}{c}
x_{0}^{4}-\frac{5}{4}x_{0}^{2}x_{1}^{2}+\frac{1}{4}x_{1}^{4}-\frac{3}{4}x_{0}^{2}x_{2}^{2}+\frac{15}{16}x_{1}^{2}x_{2}^{2}-\frac{1}{4}x_{2}^{4}-\frac{1}{70}x_{0}^{2}
\end{array}
\right)$$
and its dual ideal is calculated by {\sc Singular} as generated by the polynomial
$$\left(
\begin{array}{c}
12390875x_{0}^{12}-120264375x_{0}^{10}x_{1}^{2}+442991850x_{0}^{8}x_{1}^{4}-822808000x_{0}^{6}x_{1}^{6}+\\
827628480x_{0}^{4}x_{1}^{8}-431827200x_{0}^{2}x_{1}^{10}+91888128x_{1}^{12}+2186625x_{0}^{10}x_{2}^{2}-\\
148231125x_{0}^{8}x_{1}^{2}x_{2}^{2}+902043450x_{0}^{6}x_{1}^{4}x_{2}^{2}-1921126200x_{0}^{4}x_{1}^{6}x_{2}^{2}+\\
1725988320x_{0}^{2}x_{1}^{8}x_{2}^{2}-560787840x_{1}^{10}x_{2}^{2}-116455850x_{0}^{8}x_{2}^{4}+713525750x_{0}^{6}x_{1}^{2}x_{2}^{4}-\\
784988540x_{0}^{4}x_{1}^{4}x_{2}^{4}-703298400x_{0}^{2}x_{1}^{6}x_{2}^{4}+914535936x_{1}^{8}x_{2}^{4}+\\
232142400x_{0}^{6}x_{2}^{6}-539359800x_{0}^{4}x_{1}^{2}x_{2}^{6}-507564960x_{0}^{2}x_{1}^{4}x_{2}^{6}-\\
598014720x_{1}^{6}x_{2}^{6}-197686720x_{0}^{4}x_{2}^{8}+58816800x_{0}^{2}x_{1}^{2}x_{2}^{8}+119161344x_{1}^{4}x_{2}^{8}+\\
80183040x_{0}^{2}x_{2}^{10}+32722560x_{1}^{2}x_{2}^{10}-12753408x_{2}^{12}
\end{array}
\right)$$
You can see an illustration in figure \ref{fig:ex08_my_quartic}. 
\begin{figure}[htbp]
	\centering
		\includegraphics[width=4cm]{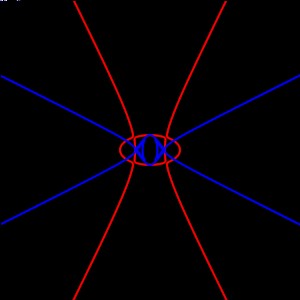}
	\caption{A quartic (red) inspired by Klein's quartic and the dual curve}
	\label{fig:ex08_my_quartic}
\end{figure}
\end{exm}
\newpage
\section{Appendix}
The code for the Steiner surface example is:
\begin{verbatim}
/////////////////////////////////////////////
// This procedure calculates the dual ideal of the homogeneous ideal id
// The output is a homogeneous ideal in the same ring
/////////////////////////////////////////////
proc dual(ideal I) {
  def R0=basering;
  if(npars(R0)>0)    {ERROR("Use a base ring without parameters!");};
  if(ord_test(R0)!=1){ERROR("The base ring must have a global ordering!");};
  if(homog(I)!=1)    {ERROR("The input ideal must be homogeneous!");};
// get some information about the base ring and the input ideal
  int n=nvars(R0);
  int m=ncols(I);
// change variables and compute transposed Jacobi matrix of I
  def NR=changevar("x()",R0); 
  setring NR;
  ideal I=fetch(R0, I);
  matrix J=transpose(jacob(I));
// adjoin auxiliary variables to the ring
  def E1=extendring(m,"l()","dp",1,NR);
  def R=extendring(n,"u()","dp",1,E1);
  setring R;
  matrix J=fetch(NR, J);
// set up system S
  ideal I=fetch(NR, I);
  matrix L=matrix([u(1..n)])-J*matrix([l(1..m)]);
  ideal S=I,L;
// eliminate first m+n variables from S by Groebner bases method
  int j,k;   
  poly prod=1;
  for(k=1;k<=n;k++){prod=prod*x(k);};
  for(j=1;j<=m;j++){prod=prod*l(j);};
  intvec v=hilb(std(S),1); 
  ideal I1=eliminate(S,prod,v);
// resubstitute variables, such that the output can be used again as input
  map f=R,(u(1..n),l(1..m),x(1..n));
  ideal I2=ideal(f(I1));
// restore initial ring and return the result
  setring R0; export R0;
  return(fetch(R, I2));
}
/////////////////////////////////////////////
LIB "surfex.lib"; // Load library
ring R1 = 0,(x(0..3)),dp; // Define ring
// First example: Steiner's Roman surface 
ideal I = (x(1)*x(2))^2+(x(1)*x(3))^2+(x(2)*x(3))^2-x(0)*x(1)*x(2)*x(3);
ideal D = dual(I);
ideal DD = dual(D);
I; D; DD; // show results
// Plot with surfex
ring R2 = 0,(x,y,z),dp;
map f=(R1, 1,x,y,z);
plotRotatedList(list(f(I), f(D)), list(x,y,z));
\end{verbatim}


\end{document}